\documentclass[12pt,oneside,reqno,fleqn]{amsart}
\usepackage{amssymb,latexsym}
\usepackage{a4wide}
\usepackage[all]{xy}

\usepackage{tikz}
\newtheorem{theorem}{Theorem}[section]

\newtheorem{corollary}[theorem]{Corollary}
\newtheorem{lemma}[theorem]{Lemma}

\newtheorem{proposition}[theorem]{Proposition}

\newtheorem{remark}[theorem]{Remark}
\newtheorem{example}[theorem]{Example}

\newtheorem{mob}[theorem]{M\"obius Invesion Formula}

\numberwithin{equation}{section}

\begin{document}

\title[Combinatorial Properties of primitive words with Non-primitive Product]{Combinatorial Properties of Primitive words with Non-primitive Product}
\author{Othman Echi \and Adel Khalfallah \and Dhaker Kroumi}

\address{Department of Mathematics and Statistics, King Fahd University of Petroleum and
	Minerals, Dhahran 31261, Saudi Arabia}
\email{othechi@yahoo.com}

\email{a$_{-}$khalfallah@yahoo.fr}

\email{kroumidhaker@yahoo.com}

\begin{abstract}
Let $\mathcal{A}$ be an alphabet of size $n\ge 2$. In this paper, we give a complete description of primitive words  $p\neq q$ over an alphabet $\mathcal{A}$ of size $n\geq2$ such that $pq$ is non-primitive and $|p|=2|q|$. In particular, if $l$ is s a positive integer, we count the cardinality of the set $\mathcal{E}(l,\mathcal{A})$ of all couples $(p,q)$ of primitive words such that $|p|=2|q|=2l$ and $pq$ is non-primitive. Then we give a combinatorial formula for this cardinality and its asymptotic behavior, as $l$ or $n$ goes to infinity.
\end{abstract}



\maketitle

\section{Introduction}

Combinatorics on words is actually an area of research focusing on combinatorial properties of words applied to formal languages. The natural environment of a word is a free semigroup. This field was the main topic of a series of books, under the collective nom de plume ``Lothaire”(see \cite{Lothaire,Lothaire2,Lothaire3}).

This field has an important role in several mathematical research areas as well as theoretical computer science  \cite{Choffrut-Karhumaki,Domosi-H rvath-2005,Domosi-Ito,Shyr-Book}. Fascinating studies connecting theory of music and combinatorics on words are also available in the literature (see for instance \cite{Clampitt}).

The present paper deals with primitive words over a nontrivial alphabet $\mathcal{A}$ (having at least two letters). The empty word over $\mathcal{A}$ will be denoted by $\varepsilon$, $\mathcal{A}^\ast$ is the set of all words over $\mathcal{A}$ and $\mathcal{A}^+$ is the set of all nonempty words over $\mathcal{A}$. A nonempty word $u$ is said to be primitive over $\mathcal{A}$, if it is not a proper power of another word (that is, if $u=v^m$ then $m=1$).  We denote by $\mathbf{Q}(\mathcal{A})$ the set of all primitive words over $\mathcal{A}$ and $\mathbf{Q}_l(\mathcal{A})$  the set of all primitive words of length $l$. Any nonempty word $u$ can be written in a unique way as a power of a primitive word, called the primitive root of $u$, denoted by $\sqrt{u}$.

 Primitive words play a crucial role in algebraic coding theory and the theory of formal languages \cite{Lothaire} and \cite{Shyr-Book}.

Whether $\mathbf{Q}(\mathcal{A})$ is a context-free language or not is a well-known long-standing open problem posed by D\"om\"osi, Horv\'ath and Ito  in  \cite{Domo-Horv-Ito} and \cite{Domosi-Ito1}. This problem was the origin of most of the combinatorial studies of primitive words.

In \cite{Reis-Shyr}, Reis and Shyr have proved  that every nonempty word which is  not a power of a letter is a product of two primitive words. So, one may think that $\mathbf{Q}(\mathcal{A})$ is ``very large" in some sense; in fact the natural density of the language of primitive words is $1$ (see \cite{Ryoma-Density}).

For  $u\in \mathcal{A} ^+$, we let
$$u^+:=\{u^n: n \textrm{ is a positive integer}\}.$$
One of the classical results about primitive words is Shyr-Yu Theorem \cite{Shyr-Yu}. It states that if $p$ and $q$ are distinct primitive words, then the language $p^+q^+$ contains at most one non-primitive word of the form $pq^m$ or $p^mq$, as $p^nq^m$ is primitive for all $n,m\geq 2$ \cite{LySch}.

In \cite{Shyr-Yu}, Shyr-Yu gave a necessary condition in order that product $pq^m$ is a $k$-power of a primitive word, with $m,k\geq 2$. To the best of our knowledge, no information about the case $m=1$ has been provided.

In \cite{Echi-RAIRO2}, necessary and sufficient conditions are provided to get $pq^m\in \mathbf{Q}^{(k)}(\mathcal{A}):=\{u^{k}:u\in\mathbf{Q}(\mathcal{A})\}$, for $m\geq 1$ and $k\geq 2$, but the expression of $p$ was not explicit.\\

The aim of this paper is to give a characterization of  primitive words $p$ and $q$ such that $pq$ is non-primitive and $|p|=2|q|$.

More precisely, the present paper is organized as follows. In Section 2, we collect necessary material useful to establish our main results. In Section~\ref{sec3}, we provide a full description of primitive words $p,q$ such that $|p|=2|q|$ and $pq$ is non-primitive. Given a positive integer $l$, we split the set $$\mathcal{E}(\mathcal{A},l)=\left\{ (p,q)\in \mathbf{Q}(\mathcal{A})^2 \colon |p|=2|q|=2l \mbox{ and } pq \not \in \mathbf{Q}(\mathcal{A}) \right\}, $$ into two disjoint subsets $\mathcal{E}_1(\mathcal{A},l)$ and $\mathcal{E}_2(\mathcal{A},l)$, according to $|\sqrt{pq}|>|q|$ and  $|\sqrt{pq}|<|q|$, respectively. The cardinalities $\varepsilon_1(n,l)$ and $\varepsilon_2(n,l)$ of these subsets are discussed. Finally, in  Section~\ref{sec4}, the asymptotic behaviors of $\varepsilon_1(n,l)$ and  $\varepsilon_2(n,l)$  are given, as $n$ or $l$ goes to $\infty$.

\section{Preliminaries}\label{sec2}For any integer $k\geq 2$, we denote by
$$\mathbf{Q}^{(k)}(\mathcal{A})=\{p^{k}:\, p\in\mathbf{Q}(\mathcal{A})\}.$$
For two words $u$ and $v$, $u$ is said to be a prefix (resp., a suffix)
of $v$ if there exists a word $x$ (resp., a word $y$) such that $v=ux$ (resp. $v=xu$).

Here, we present some well-known results useful in the sequel. In all what follows, $\mathcal{A}$ will denote an alphabet of size $n\geq 2$.

\begin{lemma}[\cite{LySch}]\label{Primitiveroot} Let $u,v$ be two nonempty words over $\mathcal{A}$. Then the following properties hold.
\begin{enumerate}
  \item $uv=vu$ if and only if there exists a word $w$ such that $u,v\in w^{+}$.
  \item There exists a unique primitive word
$\sqrt{u}$, called \emph{the primitive root} of $u$, and a unique positive
integer $\mathfrak{e}$, called \emph{the exponent} of $u$,
such that $u = \sqrt{u}^{\,  \mathfrak{e}}$.
\end{enumerate}
\end{lemma}

The following lemma is a classical result in combinatorics on words, originally due to Lyndon-Sch\"utzenberger \cite{LySch} (see also \cite{Cast-Rest-Fici,Lothaire,Restivo}). This lemma will play a crucial role in the proof of our main result.

\begin{lemma}[\cite{LySch}]\label{LySch}Let $t,v$ be two distinct nonempty words over $\mathcal{A}$ such that $tu=uv$. Then there exist a unique pair of words $(p,q)$ and a unique positive integer $m$ such that $pq$ is primitive,
$t=(pq)^m$, $v=(qp)^m$, and $u=(pq)^j p$, for some integer $j\geq 0$.\end{lemma}

\begin{lemma}[{\cite[Corollary 4]{Lent-Schuz}} and \cite{Shyr-Yu}]\label{Lent-Schuz-Shyr}
Let $u\in\mathcal{A}^{+}$ and $g,q\in \mathbf{Q}(\mathcal{A})$, with $u\notin q^+$.
If $uq^m=g^k$ for some $m,k\geq 1$, then $g\neq q$ and $\lvert g \rvert> \lvert q^{m-1} \rvert$.
\end{lemma}

\begin{lemma}[\cite{Shyr-Yu}]\label{Shyr1}Let $p\neq q\in \mathbf{Q}(\mathcal{A})$. Then the following properties hold.

\begin{enumerate}
  \item The language $p^+q^+$ contains at most one non-primitive word.
  \item If $\lvert p \rvert=\lvert q \rvert$, then $p^{+}q^{+}\setminus\{pq\}\subseteq \mathbf{Q}(\mathcal{A})$.
\end{enumerate}
\end{lemma}

It is worth noting that an alternative proof of the previous result has been given in \cite{Domosi-Horvath-Vuillon}.
In order to study local distribution of non primitive words, Shyr-Tu have established the following result.

\begin{lemma}[\cite{Shyr-Tu}]\label{Shyr-Tu}Let $u,x,y\in \mathcal{A}^{+}$ such that $|x|=|y|\leq \frac{|u|}{2}$ and $x\not=y$. Then, either $ux$ or $uy$ is a primitive word.\end{lemma}

As a direct consequence of \cite[Theorems 13, 14]{Domosi-Horvath-2006}, we obtain the following.

\begin{lemma}[Prefix-Suffix]\label{Pref-Suf}Let $q\in \mathbf{Q}(\mathcal{A})$ and $x\in \mathcal{A}^+$, with $x\neq q$. Then, the following properties hold.
\begin{enumerate}
  \item If $x$ is a prefix of $q$, then $q^kx$ is primitive for all $k\geq 2$.
  \item If $x$ is a suffix of $q$, then $xq^k$ is primitive for all $k\geq 2$.
\end{enumerate}
\end{lemma}

\begin{lemma}[\cite{Chu-Shu-Ya}]\label{Chin-Acta}Let $p \neq q\in \mathbf{Q}(\mathcal{A})$ such that $|p|=r|q|$, for some integer $r\geq 2$. Then the following properties hold.
\begin{enumerate}
  \item $pq^m$ is primitive, for all $m\geq r$.
  \item $p^mq$ is primitive for all $m\geq 2$.
\end{enumerate}
\end{lemma}

\section{Primitive Words $p\neq q$ such that $|p|=2|q|$ and $pq$ is non-primitive}\label{sec3}

We start by recalling a result giving a complete description of primitive words $p$ and $q$ such that $pq$ is non-primitive.

\begin{theorem}[\cite{Adel-Echi-Kroumi}]\label{theorem-pq}
Let $p$ and $q$ be two distinct primitive words and $k\geq 2$ be a given integer. Then, the following statements are equivalent.
\begin{enumerate}
\item $pq\in \mathbf{Q}^{(k)}(\mathcal{A})$.
\item One of the following statements holds.
\begin{enumerate}
\item $p=(xq)^{k-1}x$, with $x\in \mathcal{A}^+$ and $xq\in \mathbf{Q}(\mathcal{A})$.
\item There exist an integer $1\leq s\leq k-1$, and $\alpha,\beta\in \mathcal{A}^{+}$ such that $p=(\beta\alpha)^{k-s-1}\beta$, $q=(\alpha\beta)^s\alpha$ and  $\alpha\beta\in \mathbf{Q}(\mathcal{A})$.
\end{enumerate}
\end{enumerate}
\end{theorem}

\begin{remark}\label{remsqrt}\rm Conditions $(a)$ and $(b)$ in Theorem \ref{theorem-pq} correspond to $|\sqrt{pq}|>|q|$ and $|\sqrt{pq}|<|q|$, respectively.
\end{remark}

Let $p\neq q\in\mathbf{Q}(\mathcal{A})$. Then, following Lemma \ref{Shyr1}, the language $p^+q^+$ contains at most one non-primitive word.

We will focus on the class
$$\mathcal{C}:=\left\{(p,q)\in \mathbf{Q}(\mathcal{A})\times \mathbf{Q}(\mathcal{A})\colon pq\notin\mathbf{Q}(\mathcal{A}), |p|=2|q|\right\}.$$
First, let us identify all the elements of $\mathcal{C}$.

\begin{theorem}\label{p=2q}
Let $p$ and $q$ be two primitive words. Then the following statements are equivalent.
\begin{enumerate}
\item $|p|=2|q|$ and $pq \not \in \mathbf{Q}(\mathcal{A})$.
\item  One of the following mutually exclusive properties holds.
\begin{itemize}
    \item[$(i)$] $p=xqx$, for some $x\in \mathcal{A}^+$, with $xq \in \mathbf{Q}(\mathcal{A})$ and $|q|=2|x|$.
    \item[$(ii$)] There exist $\alpha, \beta \in \mathcal{A}^+$ and an integer $s \geq 1$ such that $\alpha \beta \in \mathbf{Q}(\mathcal{A})$,
    $$p=(\beta \alpha)^{2s}\beta\; \mbox{and} \; q=(\alpha\beta)^s \alpha, \mbox{ with } |\beta|=2|\alpha|.$$
    \item[$(iii)$] There exist $\alpha, \beta \in \mathcal{A}^+$ and an integer $s\geq 1$ such that $\alpha\beta \in \mathbf{Q}(\mathcal{A})$,
     $$p=(\beta \alpha)^{2s+1}\beta\; \mbox{and} \;q=(\alpha\beta)^s \alpha, \mbox{ with } |\alpha|=2|\beta|.$$
    \end{itemize}
\end{enumerate}
\end{theorem}
\begin{proof}\hfill

$(1) \Longrightarrow (2)$.  Assume $pq \in \mathbf{Q}^{(k)}(\mathcal{A})$, for some integer $k \geq 2$. By Theorem \ref{theorem-pq}, we consider the following two mutually exclusive cases:
\paragraph{\textit{Case 1}:} There exists $x\in \mathcal{A}^{+}$ such that $p=(xq)^{k-1}x$, where $xq\in\mathbf{Q}(\mathcal{A})$. In this case, we have
$$2|q|=|p|=(k-1)(|x|+|q|)+|x|=(k-1)|q|+k|x|.$$
Suppose that $k \geq 3$, so we get $2|q|=|p|\geq 2|q|+3|x|$, this is impossible. We conclude that $k=2$, from which we obtain $p=xqx$ and $|q|=2|x|$.
\paragraph{\textit{Case 2}:} There exist $\alpha,\beta \in \mathcal{A}^+$ and an integer $s\geq 1$, such that $p=(\beta\alpha)^{k-s-1}\beta$ and $q= (\alpha\beta)^s \alpha$, with $\alpha\beta \in \mathbf{Q}(\mathcal{A})$.
As $|p|=2|q|$, we deduce that
$$ (k-s-1)(|\alpha|+|\beta|)+|\beta|=2(s (|\alpha|+|\beta|)+|\alpha|),
$$
which leads to $(k-3s)(|\alpha|+|\beta|)= 3|\alpha|$. This forces $k-3s$ to be either $1$ or $2$, that is, $k=3s+1$ or $k=3s+2$.
\begin{itemize}
    \item If $k=3s+1$, then $p=(\beta\alpha)^{2s}\beta$ and $|\beta|=2|\alpha|$.
    \item If $k=3s+2$, then $p=(\beta\alpha)^{2s+1}\beta$ and $|\alpha|=2|\beta|$.
\end{itemize}
 \noindent $(2)\Longrightarrow (1)$. Straightforward.
\end{proof}

\medspace
Now, given a positive integer $l$, we will compute the number of all couples $(p,q)$ of primitive words such that $|p|=2|q|=2l$ and $pq \not \in \mathbf{Q}(\mathcal{A})$. Consider the set
$$\mathcal{E}(\mathcal{A},l):=\left\{ (p,q)\in \mathbf{Q}(\mathcal{A})^2 \colon |p|=2|q|=2l \mbox{ and } pq \not \in \mathbf{Q}(\mathcal{A}) \right\}, $$
we denote by $\varepsilon(n,l)$ its cardinality of this set, where $n\geq 2$ is the size of the alphabet $\mathcal{A}$.
Now, we split the above set in two disjoint sets
$$\mathcal{E}_1(\mathcal{A},l):=\{(p,q)\in \mathcal{E}(\mathcal{A},l):(p,q)\textrm{ satisfies Theorem \ref{p=2q}}(i)\}$$
and
$$\mathcal{E}_2(\mathcal{A},l):=\{(p,q)\in \mathcal{E}(\mathcal{A},l):(p,q)\textrm{ satisfies Theorem \ref{p=2q}}(ii) \textrm{ or }(iii)\}.$$
We denote by $\varepsilon_1(n,l)$ and $\varepsilon_2(n,l)$ the cardinalities of $\mathcal{E}_1(\mathcal{A},l)$ and $\mathcal{E}_2(\mathcal{A},l)$, respectively.

Define the following sets:
$$\begin{array}{lll}
\Lambda(l)\colon& =&\{d\colon d|l, d\not\equiv0\!\!\!\!\pmod{3}, d\geq 4\},\\
\Lambda^+(l)\colon& =&\{d\in \Lambda(l)\colon d \textrm{ is even}\},\\
\Lambda^-(l)\colon& =&\{d\in \Lambda(l)\colon d \textrm{ is odd}\}.
\end{array}$$

We will denote by $\pi_n(k)$ the cardinality of $\mathbf{Q}_k(\mathcal{A})$.

\begin{proposition}\label{Countp=2q-odd}
Let $l$  be a positive integer. Then, we have
 $$\varepsilon_2(n,l)=\underset{d\in \Lambda(l)}{\sum} \pi_n \left( \frac{3l}{d} \right).$$
\end{proposition}

\begin{proof} Consider the map
$$\varphi\colon \underset{d\in \Lambda(l)}{\bigcup} \mathbf{Q}_{\frac{3l}{d}}(\mathcal{A}) \longrightarrow  \mathcal{E}_2(\mathcal{A},l), $$ defined by:

\begin{itemize}
    \item[$(a)$] If $d=3s+1 \geq 4$ is a divisor of $l$ and $u=\beta\alpha \in \mathbf{Q}_{\frac{3l}{d}}(\mathcal{A})$, then we let
    $$\varphi(u)=(u^{2s}\beta, (\alpha \beta)^s \alpha), $$
    where $\beta$ is the prefix of $u$ of length $\frac{2l}{d}$ and $\alpha$ is the suffix of $u$ of length $\frac{l}{d}$.
    \item[$(b)$]  If $d=3s+2 \geq 4$ is a divisor of $l$ and $u=\beta\alpha \in \mathbf{Q}_{\frac{3l}{d}}(\mathcal{A})$, then we let
    $$\varphi(u)=(u^{2s+1}\beta, (\alpha \beta)^s \alpha), $$
    where $\beta$ is the prefix of $u$ of length $\frac{l}{d}$ and $\alpha$ the suffix of $u$ of length $\frac{2l}{d}$.
\end{itemize}
  We should verify that $\varphi$ is well defined, that is $\varphi(u)\in\mathcal{E}_2(\mathcal{A},l)$ for any $u\in\underset{d\in \Lambda(l)}{\bigcup} \mathbf{Q}_{\frac{3l}{d}}(\mathcal{A})$.
Indeed in $(a)$: Let $p=u^{2s}\beta=(\beta\alpha)^{2s}\beta$ and $q= (\alpha \beta)^s \alpha$. By Lemma \ref{Pref-Suf}, $p$ is a primitive word as $s\geq1$.
Similarly, by the same Lemma, $q$ is also a primitive word for $s \geq 2$, since $u=\beta\alpha$ and $\alpha\beta$ are conjugate and $u$ is primitive. For $s=1$, note that $q$ and $\beta \alpha^2$ are conjugate, with $|\beta|=2|\alpha|$. On the other hand, by Lemmas \ref{Pref-Suf} and \ref{Chin-Acta}, $\beta\alpha^2$ is primitive. Thus, $q$ is primitive.
It is clear that $|p|=2|q|=2l$ and $pq$ is non-primitive.  Finally, we have $\varphi(u)\in\mathcal{E}_2(\mathcal{A},l)$.
An analogous argument shows that $\varphi(u) \in \mathcal{E}_(\mathcal{A},l)$ in $(b)$.

It is clear that $\varphi$ is an onto mapping by Theorem \ref{p=2q}. It remains to show that $\varphi$ is one-to-one.
 Indeed, let $u_1,u_2\in \underset{d\in \Lambda(l)}{\bigcup} \mathbf{Q}_{\frac{3l}{d}}(\mathcal{A})$ be such that $\varphi(u_1)=\varphi(u_2)$, and  $d_1,d_2 \in \Lambda(l)$ be such that $u_1 \in \mathbf{Q}_{\frac{3l}{d_1}}(\mathcal{A})$ and $u_2\in \mathbf{Q}_{\frac{3l}{d_2}}(\mathcal{A})$. Three cases have to be considered.

If  $d_1=3s_1+1$ and $d_2=3s_2+1$, in this case, letting $u_1=\beta_1\alpha_1$  and $u_2=\beta_2\alpha_2$, with $|\beta_1|=2|\alpha_1|$ and $|\beta_2|=2|\alpha_2|$, we obtain  $(\beta_1\alpha_1)^{2s_1}\beta_1=(\beta_2\alpha_2)^{2s_2}\beta_2$ and $(\alpha_1\beta_1)^{s}\alpha_1=(\alpha_2\beta_2)^{s}\alpha_2$. Concatenating the words in the two equalities (side by side), we get $(\beta_1\alpha_1)^{3s_1+1}=(\beta_2\alpha_2)^{3s_2+1}$, and as a result, $\beta_1\alpha_1=\beta_2\alpha_2$, that is $u_1=u_2$.

The two other cases ``$d_1=3s_1+2,\; d_2=3s_2+2 $" and ``$d_1=3s_1+1,\; d_2=3s_2+2$" can be treated by analogous arguments.

It follows that $\varphi$ is a bijection, and consequently
$$
\varepsilon_2(n,l)= \underset{d\in \Lambda(l)}{\sum} \left|\mathbf{Q}_{\frac{3d}{l}}(\mathcal{A}) \right|= \underset{d\in \Lambda(l)}{\sum} \pi_n \left( \frac{3l}{d} \right).
$$
\end{proof}


\begin{proposition}\label{Countp=2q-even}
 If $l$ is an even positive integer, then
$$\varepsilon_1(n,l)=n^{\frac{l}{2}} \pi_n(l)-\underset{d\in \Lambda^+(l)}{\sum} \pi_n \left( \frac{3l}{d} \right). $$
\end{proposition}

The proof relies on the following technical lemma.


\begin{lemma}\label{technical-Comp} Let $p$ and $q$ be primitive words such that $|p|=2|q|$. Then, the following statements are equivalent.
\begin{itemize}
  \item[$(1)$] There exists $x\in \mathcal{A}^+$ such that $p=xqx$ and $xq$ is non-primitive.
  \item[$(2)$] One of the following conditions holds.
  \begin{itemize}
    \item[$(i)$] There exist $\alpha, \beta \in \mathcal{A}^+$ and an odd integer $s \geq 1$ such that $\alpha \beta \in \mathbf{Q}(\mathcal{A})$,
    $$q=(\alpha\beta)^s \alpha,\quad  p=(\beta \alpha)^{2s}\beta, \mbox{ with } |\beta|=2|\alpha|.$$
    \item[$(ii)$] There exist $\alpha, \beta \in \mathcal{A}^+$ and an even integer $s \geq 1$ such that $\alpha \beta \in \mathbf{Q}(\mathcal{A})$,
    $$q=(\alpha\beta)^s \alpha,\quad  p=(\beta \alpha)^{2s+1}\beta, \mbox{ with } |\alpha|=2|\beta|.$$
  \end{itemize}
\end{itemize}
\end{lemma}

\begin{proof}\textrm{\; }

$(1)\Longrightarrow (2)$. Assume that $p=xqx$, with $xq$ non-primitive. The fact that $|p|=2|q|$ leads to $|q|=2|x|$.
From Remark \ref{remsqrt} and Theorem \ref{p=2q}, there exist an integer $s\geq 1$, and $\alpha,\beta\in \mathcal{A}^+$ such $\alpha\beta\in \mathbf{Q}(\mathcal{A})$, $q=(\alpha\beta)^s\alpha$ and either $p=(\beta\alpha)^{2s}\beta$, with $|\beta|=2|\alpha|$;
or $p=(\beta\alpha)^{2s+1}\beta$, with $|\alpha|=2|\beta|$.

\emph{Case 1.} Suppose that $p=(\beta\alpha)^{2s}\beta$, with $|\beta|=2|\alpha|$.
As $pq=(xq)^2=(\beta\alpha)^{3s+1}$, we deduce, by Lemma \ref{Primitiveroot}, that $\sqrt{pq}=\sqrt{xq}=\beta\alpha$. Consequently, we have $xq=(\beta\alpha)^{j}$, where $j=\frac{3s+1}{2}$ is an integer. As a result, $s$ is odd.

\emph{Case 2.} Suppose that $p=(\beta\alpha)^{2s+1}\beta$, with $|\alpha|=2|\beta|$. As in Case 1, we have $xq=(\beta\alpha)^{j}$ where $j=\frac{3s+2}{2}$ is an integer. Thus, $s$ is even.

$(i)\Longrightarrow (1)$. Suppose that $s$ is odd and $$q=(\alpha\beta)^s \alpha,\quad  p=(\beta \alpha)^{2s}\beta, \mbox{ with } |\beta|=2|\alpha|.$$
Then, letting $x=(\beta\alpha)^{\frac{s-1}{2}}\beta$, we get $p=xqx$ and $xq=(\beta\alpha)^{\frac{3s+1}{2}}$ is a non-primitive word.

$(ii)\Longrightarrow (1)$. Suppose that $s$ is even and $$q=(\alpha\beta)^s \alpha,\quad  p=(\beta \alpha)^{2s+1}\beta, \mbox{ with } |\alpha|=2|\beta|.$$
Then, letting $x=(\beta\alpha)^{\frac{s}{2}}\beta$, we get $p=xqx$ and $xq=(\beta\alpha)^{\frac{3s}{2}+1}$ is a non-primitive word.
\end{proof}

\begin{proof}[Proof of Proposition \ref{Countp=2q-even}]Let $l$ be an even integer, and
\begin{align*}
&\mathcal{G}:= \left\{(p,q)\in \mathbf{Q}(\mathcal{A})^2:p=xqx, \textrm{  for some  }x\in \mathcal{A}^+, |q|=2|x|=l \textrm{  and  } xq\notin \mathbf{Q}(\mathcal{A})\right\},\\
 &\mathcal{H}:= \left\{(p,q)\in \mathbf{Q}(\mathcal{A})^2:p=xqx, \textrm{  for some  }x\in \mathcal{A}^+\textrm{  and  } |q|=2|x|=l\right\}.\\
\end{align*}

\noindent $-$ Consider the map
$$\Phi\colon \underset{d\in \Lambda^+(l)}{\bigcup} \mathbf{Q}_{\frac{3l}{d}}(\mathcal{A}) \longrightarrow  \mathcal{G}, $$ defined for $d \in \Lambda^+(l)$ and $u\in\mathbf{Q}_{\frac{3l}{d}}(\mathcal{A})$ as follows.

\begin{itemize}
    \item If  $d=3s+1$ with $s$ odd, then
    $$\Phi(u):=((\beta \alpha)^{2s}\beta, (\alpha \beta)^s \alpha),$$
     where $\beta$ is the prefix of $u$ of length $\frac{2l}{d}$ and $\alpha$ is the suffix of $u$ of length $\frac{l}{d}$.
    \item If $d=3s+2$ with $s$ even, then
    $$\Phi(u):=((\beta \alpha)^{2s+1}\beta, (\alpha \beta)^s \alpha), $$
     where $\beta$ is the prefix of $u$ of length $\frac{l}{d}$ and $\alpha$ is the suffix of $u$ of length $\frac{2l}{d}$.
\end{itemize}
Then, according to Lemma \ref{technical-Comp} and an analogous argument to the proof of Theorem \ref{Countp=2q-odd}, we deduce that $\Phi$ is a well-defined bijection. Consequently, we have
$$|\mathcal{G}|=\underset{d\in \Lambda^+(l)}{\sum} \pi_n \left( \frac{3l}{d} \right).$$

\noindent $-$ Consider the assignment $$\begin{array}{lll}
                                                    \Psi: \mathcal{A}^{\frac{l}{2}}\times\mathbf{Q}_{l}(\mathcal{A})&\longrightarrow&\mathcal{H} \\
                                                       &  & \\
                                                     \quad\quad\quad (x,q) &\longmapsto& (xqx,q).
                                                   \end{array}$$
We will show that $\Psi$ is well-defined, that is $xqx$ is primitive if $(x,q)\in \mathcal{A}^{\frac{l}{2}}\times\mathbf{Q}_{l}(\mathcal{A})$. Write $x=\sqrt{x}^{j}$, where $j\geq 1$. Then,
$qx^2=q\sqrt{x}^{2j}$, where $|q|=2|x|=2j|\sqrt{x}|$. By Lemma \ref{Chin-Acta}, $qx^2$ is primitive, so that $xqx$ is primitive. In addition, $\Psi$ is a bijection by construction.
As a result, we have $$|\mathcal{H}|=n^{\frac{l}{2}}\pi_n(l).$$

\noindent $-$ Finally, note that $\mathcal{E}_1(\mathcal{A},l)=\mathcal{H}\setminus\mathcal{G}$, from which we get
$$\varepsilon_1(n,l)=n^{\frac{l}{2}} \pi_n(l)-\underset{d\in \Lambda^+(l)}{\sum} \pi_n \left( \frac{3l}{d} \right).$$
\end{proof}

Now, we are in position to state our second main result.
\begin{theorem}
Let $l$ be a positive integer. Then, we have
$$\varepsilon(n,l)=\left\{
\begin{array}{ll}
n^{\frac{l}{2}} \pi_n(l)+\underset{d\in \Lambda^-(l)}{\sum} \pi_n \left( \frac{3l}{d} \right)& \text{if $l$ is even,} \\
\underset{d\in \Lambda(l)}{\sum} \pi_n \left( \frac{3l}{d} \right)& \text{if $l$ is odd.} \\
                                                                              \end{array}
                                                                            \right.$$
\end{theorem}


In the remainder, we give reformulations for $\varepsilon_1(n,l)$ and $\varepsilon_2(n,l)$.
Let us first recall the M\"obius Inversion Formula and a lemma from \cite{Adel-Echi-Kroumi}.
Define the M\"obius function $$\mu(n)=\left\{
\begin{array}{lll}
     1 & \text{if}& n=1 ,\\
    0&\text{if}& n \text{ is a multiple of the square of a prime number}, \\
      (-1)^k & \text{if} & n \text{ is a squarefree with } k \text{ prime factors.} \\
      \end{array} \right.$$
Clearly $\mu$ is a multiplicative function (i.e., $\mu(mn)=\mu(m)\mu(n)$, for any relatively prime numbers $m,n$).
Let $f,g\colon \mathbb{N} \longrightarrow \mathbb{C}$ be arithmetic functions, then the following property holds.
\begin{mob}[\cite{Hardy}]\rm Let $n$ be a positive integer. Then, we have the following equivalence:
$$
g(m)=\underset{d|m}{\sum}f(d)\textrm{ for all }m|n\Longleftrightarrow f(m)=\underset{d|m}{\sum}\mu(d)g\left(\frac{m}{d}\right)\textrm{ for all }m|n.$$
\end{mob}


\begin{lemma}[\cite{Adel-Echi-Kroumi}]\label{lemma1}Let  $l$ be a positive integer and $\mathbf{r}$ be a prime number such that $gcd(l,\mathbf{r})=1$. Then, for any nonnegative integer $m$, we have
$$
\pi_n(\mathbf{r}^{m+1}l)=\underset{d|l}{\sum}\mu(d)\left(n^{\frac{\mathbf{r}^{m+1}l}{d}}-n^{\frac{\mathbf{r}^{m}l}{d}}\right).$$
\end{lemma}
Combining this lemma and the M\"obius inversion formula, we get the following corollary.

\begin{corollary}\label{corollary1}
Under the same notations as conditions in Lemma \ref{lemma1}, we have
 $$
 \sum_{d|l} \pi_n\left( \frac{\mathbf{r}^{m+1}l}{d}\right)=n^{\mathbf{r}^{m+1}l}-n^{\mathbf{r}^m l}.
 $$
\end{corollary}


\begin{proposition}\label{prop1}
Let $l=3^{m}l_1$ be a positive integer, where  $l_1\geq 2$, $gcd(3,l_1)=1$ and $m\geq 0$. Then, we have
$$
\varepsilon_2(n,l)=
  \begin{cases}
   n^{3l}-n^l-\pi_n(3l)  & \text{if $l_1$ is odd,} \\
   n^{3l}-n^l-\pi_n(3l)-\pi_n(\frac{3l}{2}) & \text{if $l_1$ is even}.
  \end{cases}
$$
\end{proposition}

\begin{proof}
According to Corollary \ref{corollary1}, we obtain
 $$\sum_{d|l_1} \pi_n\left( \frac{3^{m+1}l_1}{d}\right)=n^{3^{m+1}l_1}-n^{3^m l_1}.$$
Inserting the last relation in the expression of $\varepsilon_2(n,l)$ given by Proposition \ref{Countp=2q-odd}, we get
\begin{equation*}
\varepsilon_2(n,3^ml_1)=\sum_{d|l_1} \pi_n\left( \frac{3^{m+1}l_1}{d}\right) -\pi_n\left(3^{m+1}l_1\right)=n^{3l}-n^l-\pi_n(3l),
\end{equation*}
if $l_1$ is odd, and
$$
\varepsilon_2(n,3^ml_1)=\sum_{d|l_1} \pi_n\left( \frac{3^{m+1}l_1}{d}\right) -\pi_n(3^{m+1}l_1)-\pi_n\left(\frac{3^{m+1}l_1}{2}\right)=n^{3l}-n^l-\pi_n(3l)-\pi_n\left(\frac{3l}{2}\right)\!,$$
if $l_1$ is even.
\end{proof}


\begin{proposition}\label{prop2}
Let $l=3^{m}l_1$ be an even positive integer, with $gcd(3,l_1)=1$ and $m\geq 0$. Then, we have
$$
\varepsilon_1(n,l)=
  n^{\frac{l}{2}}\left(\pi_n(l)+1\right)+ \pi_n\left(\frac{3l}{2}\right)- n^{\frac{3l}{2}}.
$$
\end{proposition}

\begin{proof}
We have
\begin{align*}
\sum_{d\in\Lambda^{+}(3^{m}l_1)}\pi_n\left( \frac{3^{m+1}l_1}{d}\right)&=\sum_{d_1|\frac{l_1}{2}}\pi_n\left( \frac{3^{m+1}l_1}{2d_1}\right)-\pi_n\left( \frac{3^{m+1}l_1}{2}\right)\!.
\end{align*}
Now, using Corollary \ref{corollary1}, we obtain
\begin{align*}
\sum_{d\in\Lambda^{+}(3^{m}l_1)}\pi_n\left( \frac{3^{m+1}l_1}{d}\right)&=n^{3^{m+1}\frac{l_1}{2}}-n^{3^m\frac{l_1}{2}}-\pi_n\left( \frac{3^{m+1}l_1}{2}\right)\!.
\end{align*}
Finally, using Proposition \ref{Countp=2q-even}, we get the desired formula.
\end{proof}\medskip
Now, combining Propositions \ref{prop1} and \ref{prop2}, we have the following.
\begin{theorem}
Let $l=3^{m}l_1$ be an even positive integer, where  $l_1\geq 2$, $gcd(3,l_1)=1$ and $m\geq 0$. Then, we have
$$
\varepsilon(n,l)=
  \begin{cases}
   n^{3l}-n^l-\pi_n(3l)  & \text{if $l_1$ is odd,} \\
   n^{3l}+ n^{\frac{l}{2}}\left(\pi_n(l)+1\right)-n^l-\pi_n(3l)- n^{\frac{3l}{2}}& \text{if $l_1$ is even}.
  \end{cases}
$$
\end{theorem}\medskip


Now, we will give combinatorial forms of $\varepsilon_1(n,l)$ and  $\varepsilon_2(n,l)$. First,
for a set of positive integers $L=\{n_1,\ldots,n_m\}$, define
$$\mathfrak{p}(L)=\prod_{i=1}^{m}n_i,$$
with the convention $\mathfrak{p}(\emptyset)=1$.
If $l$ is a positive integer, then we denote by $\mathfrak{pf}(l)$ the set of all prime factors of $l$.

\begin{theorem}\label{theorem}
Let $l=3^m2^sl_1$, where $gcd(3,l_1)=gcd(2,l_1)=1$. Then the following properties hold.
\begin{enumerate}
    \item If $s\geq2$, then
    $$\varepsilon_2(n,l)=\underset{L\in \Gamma_1(l)}{\sum} (-1)^{|L|+1} n^{\frac{3l}{\mathfrak{p}(L)}},$$
    where $\Gamma_1(l)=\left\{L \subseteq \{3,4\}\cup \mathfrak{pf}(l_1)\colon (L\neq\emptyset)\wedge(L\neq\{3\})\right\}$.
    \item If $s\leq 1$, then
    $$\varepsilon_2(n,l)=\sum_{L\in \Gamma_2(l)} (-1)^{|L|+1} n^{\frac{3l}{\mathfrak{p}(L)}},$$ where
    $\Gamma_2(l):=\left\{L\subseteq \{3\}\cup \mathfrak{pf}(l_1)\colon (L\neq\emptyset)\wedge(L\neq\{3\})\right\}.$
\end{enumerate}
\end{theorem}

\begin{proof}\textrm{\; }

\noindent $(1)$ Suppose that $s\geq2$. According to Proposition \ref{prop1}, we have
$$\varepsilon_2(n,l)=
n^{3l}-n^l-\pi_n(3l)-\pi_n(\frac{3l}{2}).
$$
Now, applying Lemma \ref{lemma1} to $\pi_n(3l)$ and $\pi_n\left(\frac{3l}{2}\right)$, we obtain
\begin{align*}
\pi_n(3l)&=\sum_{d|2^{s}l_1}\mu(d)\left(  n^{\frac{3l}{d}}- n^{\frac{l}{d}}\right)\\
&=\sum_{d|l_1}\mu(d)\left(  n^{\frac{3l}{d}}- n^{\frac{l}{d}}\right)+\sum_{d|l_1}\mu(2d)\left(  n^{\frac{3l}{2d}}- n^{\frac{l}{2d}}\right)\\
&=\sum_{d|l_1}\mu(d)\left(  n^{\frac{3l}{d}}- n^{\frac{l}{d}}\right)-\sum_{d|l_1}\mu(d)\left(  n^{\frac{3l}{2d}}- n^{\frac{l}{2d}}\right)
\end{align*}
and
\begin{align*}
\pi_n\left(\frac{3l}{2}\right)&=\pi_n(3^{m+1}2^{s-1}l_1)= \sum_{d|l_1}\mu(d)\left(  n^{\frac{3l}{2d}}- n^{\frac{l}{2d}}\right)-\sum_{d|l_1}\mu(d)\left(n^{\frac{3l}{4d}}- n^{\frac{l}{4d}}\right).
\end{align*}
As a result,
\begin{align*}
\varepsilon_2(n,l)&=-\underset{d|l_1,d\not =1}{\sum}\mu(d)\left(n^{\frac{3l}{d}}- n^{\frac{l}{d}}\right)+\sum_{d|l_1}\mu(d)\left(n^{\frac{3l}{4d}}- n^{\frac{l}{4d}}\right) \\
&= - \underset{d|l_1,d\not =1}{\sum} \mu(d) n^{\frac{3l}{d}}+\underset{d|l_1,d\not =1}{\sum} \mu(d) n^{\frac{3l}{3d}}+\underset{d|l_1}{\sum}\mu(d) n^{\frac{3l}{4d}}-\underset{d\mid l_1}{\sum} \mu(d) n^{\frac{3l}{12d}}\!.
\end{align*}
Every term of the previous sum can be written in a combinatorial form as follows
$$\begin{array}{rll}
 -\underset{d|l_1,\; d\not =1}{\sum} \mu(d) n^{\frac{3l}{d}}&= &\underset{3,4\notin L\in \Gamma_1(l)}{\sum}(-1)^{|L|+1} n^{\frac{3l}{\mathfrak{p}(L)}},\\
\underset{d|l_1,\; d\not =1}{\sum} \mu(d) n^{\frac{3l}{3d}}&=& \underset{3\in L,4\notin L\in \Gamma_1(l)}{\sum} (-1)^{|L|+1} n^{\frac{3l}{\mathfrak{p}(L)}}, \\
\underset{d|l_1}{\sum} \mu(d) n^{\frac{3l}{4d}}&=& \underset{4\in L,3\notin L\in \Gamma_1(l)}{\sum} (-1)^{|L|+1} n^{\frac{3l}{\mathfrak{p}(L)}}, \\
-\underset{d|l_1}{\sum} \mu(d) n^{\frac{3l}{12d}}&=& \underset{3,4\in  L\in \Gamma_1(l)}{\sum} (-1)^{|L|+1} n^{\frac{3l}{\mathfrak{p}(L)}}.\\
\end{array}$$
Therefore, $\varepsilon_2(n,l)=\underset{L\in \Gamma_1(l)}{\sum} (-1)^{|L|+1} n^{\frac{3l}{\mathfrak{p}(L)}}$.

\noindent $(2)$ Suppose that $s\leq 1$, then two cases have to be considered.
\begin{itemize}
\item Assume that $s=1$.  Hence
$$\pi_n\left(\frac{3l}{2}\right)=\sum_{d|l_1} \mu(d)\left(  n^{\frac{3l}{2d}}- n^{\frac{l}{2d}}\right).$$
Thus, the expression $\varepsilon_2(n,l)$ is equal to
\begin{align*}
&n^{3l}-n^l -\sum_{d|l_1}\mu(d)\left( n^{\frac{3l}{d}}- n^{\frac{l}{d}}\right)+\sum_{d|l_1}\mu(d)\left(  n^{\frac{3l}{2d}}- n^{\frac{l}{2d}}\right)-\sum_{d|l_1} \mu(d)\left(  n^{\frac{3l}{2d}}- n^{\frac{l}{2d}}\right)\\
 =&\,n^{3l}-n^l -\sum_{d|l_1}\mu(d)\left( n^{\frac{3l}{d}}- n^{\frac{l}{d}}\right)\\
 =&-\underset{d\not=1,d|l_1}{\sum} \mu(d)n^{\frac{3l}{d}}+\underset{d\not=1,d|l_1}{\sum} \mu(d)n^{\frac{l}{d}}\\
 =&\underset{3\not\in L\in \Gamma_2(l)}{\sum} (-1)^{|L|+1} n^{\frac{3l}{\mathfrak{p}(L)}}+\underset{3\in L\in \Gamma_2(l)}{\sum} (-1)^{|L|+1} n^{\frac{3l}{\mathfrak{p}(L)}}\\
 =&\underset{L\in \Gamma_2(l)}{\sum} (-1)^{|L|+1} n^{\frac{3l}{\mathfrak{p}(L)}}.
\end{align*}
\item Now, assume that $s=0$. In this case, we have
$$\varepsilon_2(n,l)= n^{3l}-n^{l}-\pi_n(3l)
=-\sum_{d|l_1,d\not= 1}\mu(d)\left( n^{\frac{3l}{d}}- n^{\frac{l}{d}}\right)
=\underset{L\in \Gamma_2(l)}{\sum} (-1)^{|L|+1} n^{\frac{3l}{\mathfrak{p}(L)}}.
$$\end{itemize}\end{proof}

\begin{theorem}\label{theorem1}Let $l=3^m2^sl_1$ be an even positive integer, where $gcd(3,l_1)=gcd(2,l_1)~=~1$. Then, we have
    $$\varepsilon_1(n,l)=\underset{L\subseteq\mathfrak{pf}(l)}{\sum} (-1)^{|L|} n^{\frac{l}{\mathfrak{p}(L)}+\frac{l}{2}}+\underset{L\in \Gamma_2(l)}{\sum} (-1)^{|L|} n^{\frac{l}{2\mathfrak{p}(L)}}.$$
\end{theorem}

\begin{proof}Using Lemma \ref{lemma1} and Proposition \ref{prop2}, we have
\begin{align*}
\varepsilon_1(n,l)=&\, n^{\frac{l}{2}}\left(\pi_n(l)+1\right)+ \pi_n\left(\frac{3l}{2}\right)- n^{\frac{3l}{2}}\\
=&\,n^{\frac{l}{2}}\pi_n(l)-\left(n^{\frac{3l}{2}}-n^{\frac{l}{2}}\right)+\sum_{d|l_1}\mu(d)\left( n^{\frac{3l}{d}}- n^{\frac{l}{d}}\right)\\
=&\,n^{\frac{l}{2}}\sum_{d|l}\mu(d)n^{\frac{l}{d}}+\sum_{d|l_1,d\not=1}\mu(d)\left( n^{\frac{3l}{d}}- n^{\frac{l}{d}}\right)\\
=&\underset{L\subseteq\mathfrak{pf}(l)}{\sum} (-1)^{|L|} n^{\frac{l}{\mathfrak{p}(L)}+\frac{l}{2}}+\underset{L\in \Gamma_2(l)}{\sum} (-1)^{|L|} n^{\frac{l}{2\mathfrak{p}(L)}}.
\end{align*}
\end{proof}


\begin{example}\rm Let $m\geq0$ and $s\geq1$ be integers. The following examples are direct consequences of the combinatorial forms in Theorems \ref{theorem} and \ref{theorem1}.
\begin{enumerate}
\item We have
$$\varepsilon_1(n,3^m2^s)=n^{\frac{3l}{2}}+n^{\frac{2l}{3}}-n^{l}-n^{\frac{5l}{6}}$$
and
$$
\varepsilon_2(n,3^m2^s)=
  \begin{cases}
   0  & \text{if } s=1, \\
   n^{3^{m+1}2^{s-2}}- n^{3^{m}2^{s-2}}& \text{if }s\geq2.
  \end{cases}
$$
\item Let $p\geq5$ be prime. Then, we have $\varepsilon_1(n,3^mp^s)=0$
and
$$
\varepsilon_2(n,3^mp^s)=  n^{3^{m+1}p^{s-1}}- n^{3^{m}p^{s-1}}.
$$
\end{enumerate}
\end{example}


\section{Asymptotic behaviors of $\varepsilon_1(n,l)$ and $\varepsilon_2(n,l)$}\label{sec4}

Let $l$ be a positive integer such that $\Lambda(l)\not=\emptyset$. It is clear, from the combinatorial description of $\varepsilon_2(n,l)$, that
$$
\varepsilon_2(n,l) \sim n^{\frac{3l}{\delta(l)}} \mbox{ as } n \to \infty,
$$
where $\delta(l)=\min\Lambda(l)$. Let $l$ be an even number. Then, according to the combinatorial form of
$\varepsilon_1(n,l)$,  we have
$$\varepsilon_1(n,l) \sim n^{\frac{3l}{2}}\mbox{ as } n \to \infty.$$

Another point of interest is the asymptotic behaviors of $\varepsilon_1(n,l)$ and $\varepsilon_2(n,l)$, as $l\rightarrow\infty$ for a fixed $n$.
Let us recall the following.
\begin{lemma}[\cite{Adel-Echi-Kroumi}]\label{lemma4}
We have $\pi_n(l)\sim n^l$, as $l$ goes to $\infty$.
\end{lemma}

\begin{proposition}\label{proposition5}
	As $l$ is even and goes to infinity, we have
	$$
	\varepsilon_1(n,l)\sim n^{\frac{3l}{2}}.
	$$
\end{proposition}


\begin{proof}
	Note that if $d\in\Lambda^{+}(l)$, then we have $4\leq d\leq l$. This leads to
	$$
	\sum_{d\in \Lambda^+(l)}\pi_{n}\left(\frac{3l_k}{d}\right)
	\leq\sum_{d\in \Lambda^+(l)}\pi_{n}\left(\frac{3l}{4}\right)\leq l\pi_{n}\left(\frac{3l}{4}\right)\leq n^{\frac{3l}{4}}l.
	$$
	Thus
	$$
	\sum_{d\in \Lambda^+(l)}\pi_{n}\left(\frac{3l}{d}\right)=o\left(n^{\frac{3l}{2}}\right)\!.
	$$
	Using Lemma \ref{lemma4}, we get
	$$
	\varepsilon_1(n,l)=n^{\frac{l}{2}} \pi_n(l)-\underset{d\in \Lambda^+(l)}{\sum} \pi_n \left( \frac{3l}{d} \right)\sim n^{\frac{3l}{2}}\!.
	$$
\end{proof}

\begin{theorem} We have
$$
\varepsilon_2(n,l)= \mathcal{O}(n^{\frac{3l}{4}}) \mbox{ as } l \to \infty.
$$
\end{theorem}
\begin{proof}
 According to Proposition \ref{Countp=2q-odd}, we have
$$\varepsilon_2(n,l)=\underset{d\in \Lambda(l)}{\sum}\pi_n \left( \frac{3l}{d} \right) \leq \sum_{1\leq i \leq \frac{l}{\delta(l)}} \pi_n(3i),$$
where $\delta(l)=\min\Lambda(l)\geq4$.
On the other hand, by using Lemma \ref{Shyr-Tu}, we obtain
$$\pi_n(l+3)\geq (n^2-2)\pi_n(l+1)\geq (n^2-2)(n-1)\pi_n(l)\geq 2 \pi_n(l),$$
which yields
$\pi_n(l) \leq \frac{1}{2^i}\pi_n(l+3i),$ for any nonnegative integer $i\geq0$.
Therefore, we get
$$
\varepsilon_2(n,l)\leq \left(1+\frac{1}{2}+\ldots+\frac{1}{2^{\frac{l}{\delta(l)} -1}}\right)\pi_n\left(\frac{3l}{\delta(l)}\right) \leq 2 \pi_n\left(\frac{3l}{\delta(l)}\right)\leq 2n^{\frac{3l}{\delta(l)}}\leq 2n^{\frac{3l}{4}}.
$$
As a result, we have $\varepsilon_2(n,l)=\mathcal{O}(n^{\frac{3l}{4}})$, as $l\to \infty$.
\end{proof}

For some specific values of $l$, we can provide the following asymptotic behavior of $\varepsilon_2(n,l)$.

\begin{proposition}\label{proposition3}If $\delta(l)$ is bounded, then
$$
\varepsilon_2(n,l)\sim n^{\frac{3l}{\delta(l)}}, \textrm{ as }l\to\infty.
$$
\end{proposition}


\begin{proof}
We have
$$
\varepsilon_2(n,l)=\pi_{n}\left(\frac{3l}{\delta(l)}\right)+\sum_{d\in \Lambda(l)\setminus\{\delta(l)\}}\pi_{n}\left(\frac{3l}{d}\right)\!,
$$
and
$$
\sum_{d\in \Lambda(l)\setminus\{\delta(l)\}}\pi_{n}\left(\frac{3l}{d}\right)
\leq\sum_{d\in \Lambda(l)\setminus\{\delta(l)\}}\pi_{n}\left(\frac{3l}{\delta(l)+1}\right)\leq l\pi_{n}\left(\frac{3l}{\delta(l)+1}\right)\!.
$$
As $\delta(l)$ is bounded, using  Lemma~\ref{lemma4}, one may check easily that $l\pi_{n}\left(\frac{3l}{\delta(l)+1}\right)=o\left( n^{\frac{3l}{\delta(l)}}\right)$, completing the proof.
\end{proof}
\begin{remark}\rm The assumption ``$\delta(l)$ is bounded" is essential in the previous proposition, as illustrated below.
\end{remark}

Before stating further asymptotic behavior of $\varepsilon_2(n,l)$, let us review some facts about gaps between consecutive primes (see \cite{Ford-Tao2,Ford-Tao3,Maynard1,Maynard2}). For any positive integer $n$, we denote by $p_n$ the $n-$th prime and $g(n):=p_{n+1}-p_n$ the gap between consecutive primes.

First, it is quite clear that $\underset{n\longrightarrow \infty}{\limsup}\; g(n)=+\infty$. Indeed, for every positive integer $n$, the interval $[n!+2, n! +n]$ does not contain any prime number.

Small gaps between consecutive primes are related to the ``twin prime conjecture", which states that there are infinitely many primes $p$ such that $p+2$ is also prime. This could also be stated by saying that $g(n)=2$ occurs for infinitely many $n$.

In \cite{Maynard1}, Maynard  gave a significant result proving that there are infinitely many prime gaps smaller than some given constant. More precisely, he showed that
$$g:=\underset{n\longrightarrow \infty}{\liminf} \; g(n)<600.$$ This result was an improvement of that of Zhang \cite{Zhang}.

\begin{proposition}\label{proposition4}
 Consider $l_k=p_kp_{k+1}$ for $k\in\mathbb{N}$. Then
 $$\varepsilon_2(n,l_k)\sim n^{\frac{3l_k}{\delta(l_k)}}\left(n^{-3g(k)}+1\right),\textrm{ as }k\to\infty.$$
\end{proposition}

\begin{proof}For $k\geq3$, according to Theorem \ref{theorem}, we have
$$
\varepsilon_2(n,l_{k})=n^{3p_{k+1}}+n^{3p_k}+n-n^{p_{k+1}}-n^{p_k}-n^3.
$$
Thus
$$
\varepsilon_2(n,l_{k})\sim n^{3p_{k+1}}\left(1+n^{-3(p_{k+1}-p_k)}\right)=n^{\frac{3l_{k}}{\delta(l_{k})}}\left(1+n^{-3g(k)}\right),\textrm{  as  }k\to\infty.$$
\end{proof}

Let $(\varphi(k), k\in \mathbb{N})$ and $(\psi(k), k\in \mathbb{N})$ be two increasing sequences of positive integers such that $$\underset{k\longrightarrow \infty}{\lim}g(\varphi(k))=\infty\textrm{ and }\underset{k\longrightarrow \infty}{\lim}g(\psi(k))=g.$$ As an immediate consequence of the above proposition, we obtain the following.

\begin{corollary}\label{cor-c} As $k$ goes to $\infty$, we have
$$\varepsilon_2(n,l_{\varphi(k)})\sim n^{\frac{3l_{\varphi(k)}}{\delta(l_{\varphi(k)})}}, \; \; \varepsilon_2(n,l_{\psi(k)})\sim n^{\frac{3l_{\psi(k)}}{\delta(l_{\psi(k)})}}\left(n^{-3g}+1\right).$$
\end{corollary}




\begin{thebibliography}{KLR73}



\bibitem{Cast-Rest-Fici}
\newblock G. Castiglione,  G.Fici and A. Restivo,
\newblock Primitive sets of words,
 \newblock  \emph{Theoret. Comput. Sci.}  $\mathbf{866}\; (2021),$ 25--36

\bibitem{Choffrut-Karhumaki}
\newblock C. Choffrut and J. Karhumaki,
\newblock Combinatorics of Words,
\newblock  \emph{Handbook of Formal Languages} vol. 1, Springer-Verlag, Berlin, Heidelberg, 1997, pp. $329-438$.

\bibitem{Chu-Shu-Ya}
\newblock C. Chunhua, Y. Shuang and D. Yang,
\newblock Some kinds of primitive and non-primitive words,
\newblock  \emph{Acta Inform.} $\mathbf{51}\; (2014),$ 339--346.

\bibitem{Clampitt}
\newblock D. Clampitt, T. Noll,
\newblock Naming and ordering the modes, in light of combinatorics on words,
\newblock  \emph{J. Math. Music} $\mathbf{12}\; (2018),$ 134--153.

\bibitem{Domo-Horv-Ito}
P. D\"om\"osi, S. Horv\"ath and M. Ito,
\newblock Formal languages and primitive words,
\newblock  \emph{Publ. Math. Debrecen} $\mathbf{42}\; (1993),$ 315–-321.

\bibitem{Domosi-Horvath-2006}
\newblock P. D\"om\"osi and G. Horv\'ath,
\newblock Alternative proof of the Lyndon-Sch\"utzenberger theorem,
\newblock  \emph{Theoret. Comput. Sci.} $\mathbf{366}\; (2006),\; 194-198$.

\bibitem{Domosi-H rvath-2005}
P. D\"om\"osi and G. Horvath,
\newblock The language of primitive words is not regular: two simple proofs,
\newblock  \emph{Bull. Eur. Assoc. Theor. Comput. Sci. EATCS}  $\mathbf{87}\;  (2005),\;  191-197$.

 \bibitem{Domosi-Horvath-Vuillon}
 \newblock P.D\"om\"osi,  G. Horv\'ath and L. Vuillon,
 \newblock  On the Shyr-Yu theorem,
 \newblock  \emph{Theoret. Comput. Sci.}  $\mathbf{410}\; (2009),\;  4874-4877$.

\bibitem{Domosi-Ito1}
\newblock P. D\"om\"osi and  M. Ito,
\newblock Context-free Languages and Primitive Words,
\newblock \emph{World Scientific Publishing  Co. Pte. Ltd.}, Hackensack, NJ (2015).

\bibitem{Domosi-Ito}
\newblock P. D\"om\"osi, M. Ito and S. Marcus,
\newblock Marcus contextual languages consisting of primitive words,
\newblock \emph{Discrete Math.} $\mathbf{308}\; (2008),\;  4877-4881$.

\bibitem{FGKT2016}
\newblock K. Ford, B. Green, S. Konyagin and T. Tao,
\newblock Large gaps between consecutive prime numbers,
\newblock \emph{Ann. of Math.} $\mathbf{183}\; (2016),\; 935-974$.

\bibitem{Echi-RAIRO2}
\newblock O. Echi,
\newblock Non-primitive words of the form $pq^m$,
\newblock \emph{RAIRO Theor. Inform. Appl.} $\mathbf{51}\; (2017),$ 135–139.

\bibitem{Adel-Echi-Kroumi}
\newblock O. Echi, A. Khalfallah and D. Kroumi,
\newblock On Primitive words with non-primitive product,
\newblock \emph{RAIRO Theor. Inform. Appl.} (to appear).

\bibitem{Ford-Tao2}
\newblock K. Ford, B. Green, S. Konyagin, J. Maynard, T. Tao,
\newblock Long gaps between primes,
\newblock \emph{J. Amer. Math. Soc.} $\mathbf{31}\; (2018)\; 65-105$.

\bibitem{Ford-Tao3}
\newblock K. Ford, B. Green, S. Konyagin, T. Tao,
\newblock Large gaps between consecutive prime numbers,
\newblock \emph{Ann. of Math.} (2) $\mathbf{183}\; (2016)\; 935-974$.


\bibitem{Fine-Wilf}
\newblock N. J. Fine and H. S. Wilf,
\newblock Uniqueness theorems for periodic functions
\newblock \emph{Proc. Amer. Math. Soc.}  $\mathbf{16}\; (1965),\; 109-114.$


\bibitem{Hardy}
\newblock G.H. Hardy and E.M. Wright,
\newblock An Introduction to the Theory of Numbers,
\newblock \emph{Oxford University Press}, Oxford, Sixth edition (2008).

\bibitem{Horv-Densty}
\newblock S. Horv\'ath,
\newblock Strong interchangeability and nonlinearity of primitive words,
\newblock \emph{Proc. Algebraic Methods in Language Processing}, 1995, Univ. of Twente, Enschede, The Netherlands, 6–8 December, 1995, Univ. Twente Service Centrum,  1995, pp. $173-178$.

\bibitem{Lent-Schuz}
\newblock A. Lentin and  M. P. Sch\"utzenberger,
\newblock A combinatorial problem in the theory of free monoids,
\newblock \emph{1969 Combinatorial Mathematics and its Applications (Proc. Conf., Univ. North Carolina, Chapel Hill, N.C., 1967)} pp. $128-144$ Univ. North Carolina Press, Chapel Hill, N.C.

\bibitem{Lothaire}
\newblock M. Lothaire,
\newblock Combinatorics on Words,
\newblock \emph{Cambridge Mathematical Library},  Cambridge University Press, Cambridge, 1997.  ISBN: 0-521-59924-5.

\bibitem{Lothaire2}
\newblock M. Lothaire,
\newblock Applied Combinatorics on Words,
\newblock \emph{Encyclopedia of Mathematics and its Applications}, 105. Cambridge University Press, Cambridge, 2005.

\bibitem{Lothaire3}
\newblock M. Lothaire,
\newblock Algebraic Combinatorics on Words,
\newblock \emph{Encyclopedia of Mathematics and its Applications}, 90. Cambridge University Press, Cambridge, 2002.

\bibitem{LySch}
R. C. Lyndon  and  M. P. Sch\"utzenberger
\newblock The equation $a^M = b^Nc^P$ in a free group
\newblock \emph{Mich. Math. J.} $\mathbf{9}\; (1962),\; 289-98$.


\bibitem{Maynard1}
\newblock J. Maynard,
\newblock Small gaps between primes
\newblock \emph{Ann. of Math.}   $\mathbf{181}\; (2015),\; 383-413$.

\bibitem{Maynard2}
\newblock J. Maynard
\newblock Large gaps between primes,
\newblock \emph{Ann. of Math.} $\mathbf{183}\; (2016)\; 915-933$.

\bibitem{Restivo}
\newblock A. Restivo,
\newblock On a question of McNaughton and Papert,
\newblock \emph{Inf. Control}, $\mathbf{25}\; (1974),\; 93-101$.


\bibitem{Reis-Shyr}
\newblock C. Reis and H.J. Shyr,
\newblock Some properties of disjunctive languages on a free monoid,
\newblock \emph{Inf. Control}, $\mathbf{37}\; (1978),$ 334–344.

\bibitem{Ryoma-Density}
\newblock Ryoma Si\'nya,
\newblock Asymptotic Approximation by Regular Languages
SOFSEM 2021: theory and practice of computer science,
\newblock \emph{Lecture Notes in Comput. Sci.} $\mathbf{12607}\; (2021)$ Springer.

\bibitem{Shyr-Yu}
\newblock H.J. Shyr and S.S. Yu,
\newblock Non-primitive words in the language $p^+q^+$,
\newblock \emph{Soochow J. Math.} $\mathbf{20}\; (1994),\; 535-546$.

\bibitem{Shyr-Book}
\newblock H.J .Shyr,
\newblock Free Monoids and languages,
\newblock \emph{Lecture Notes}. Second edition. Hon Min Book Co., Taichung, 1991


\bibitem{Shyr-Tu}
\newblock H.J. Shyr and F.K. Tu,
\newblock Local distribution of non-primitive words,
\newblock \emph{Ordered structures and algebra of computer languages}, World Scientific (1993), pp $202-217$.

\bibitem{Zhang}
\newblock Y. Zhang,
\newblock Bounded gaps between primes,
\newblock \emph{Ann. of Math.} $\mathbf{179}\; (2014),$ 1121-1174.


\end{thebibliography}
\end{document}